\documentclass[reqno,12pt]{amsart}
\usepackage{amscd}

\usepackage{mathrsfs}
\usepackage{indentfirst,latexsym,bm,amsmath,amssymb,amsthm,amscd}
\usepackage{diagrams}
\usepackage[all]{xy}
\newtheorem{theorem}{Theorem}[section]
\newtheorem{lemma}[theorem]{Lemma}
\newtheorem{proposition}[theorem]{Proposition}
\newtheorem{corollary}[theorem]{Corollary}

\theoremstyle{definition}
\newtheorem{definition}[theorem]{Definition}
\newtheorem{example}[theorem]{Example}

\numberwithin{equation}{section}

\begin{document}

\title
{On abelian canonical n-folds of general type}

\author[Rong Du]{Rong Du$^{\dag}$}
\address{Department of Mathematics\\
Shanghai Key Laboratory of PMMP\\
East China Normal University\\
Rm. 312, Math. Bldg, No. 500, Dongchuan Road\\
Shanghai, 200241, P. R. China} \email{rdu@math.ecnu.edu.cn}

\author[Yun Gao]{Yun Gao$^{\dag\dag}$}

\address{Department of Mathematics, Shanghai Jiao Tong University,
Shanghai 200240, P. R. of China}
\email{gaoyunmath@sjtu.edu.cn}

\thanks{$^{\dag}$ The Research Sponsored by the National Natural Science Foundation of China (Grant No. 11471116) and Science and Technology Commission of Shanghai Municipality (Grant No. 13dz2260400).}
\thanks{$^{\dag\dag}$The Research Sponsored by the National Natural Science Foundation of China (Grant No. 11271250,11271251) and SMC program of Shanghai Jiao Tong University.}
\thanks{Both authors are supported by China NSF (Grant No. 11531007)}

 \maketitle

 \begin{abstract}{Let $X$ be a Gorenstein minimal projective $n$-fold with at worst locally factorial terminal singularities, and suppose that the canonical map of $X$ is generically finite onto its image. When $n<4$, the canonical degree is universally bounded. While the possibility of obtaining a universal bound on the canonical degree of $X$ for $n \geqslant 4$ may be inaccessible, we give a uniform upper bound for the degrees of certain abelian covers. In particular, we show that if the canonical divisor $K_X$ defines an abelian cover over $\mathbb{P}^n$, i.e., when $X$ is an \emph{abelian canonical $n$-fold}, then the canonical degree of $X$ is universally upper bounded by a constant which only depends on $n$ for $X$ non-singular. We also construct two examples of non-singular minimal projective $4$-folds of general type with canonical degrees $81$ and $128$.
}\end{abstract}

\section{Introduction}\label{secintro}
The study of the canonical maps of $n$-dimensional projective varieties of general type is one of the central problems in algebraic geometry. For non-singular algebraic surfaces, Beauville (\cite{Bea}) proved that the degree of the canonical map is less than or equal to $36$ and that equality holds if and only if $X$ is a ball quotient surface with $K_X^2=36$, $p_g=3$, $q=0$, and $|K_X|$ is base point free. For $n=3$, Chen posed an open problem in \cite{Ch} as follows. Let $X$ be a minimal projective $3$-fold with at worst locally factorial terminal singularities, and suppose that the canonical map is generically finite onto its image. Is the generic degree of the canonical map universally bounded from above? Later, Hacon (\cite{Ha}) gave some examples of 3-folds of general type with terminal singularities such that the canonical degrees of these 3-folds can be arbitrarily large, but he also showed that the answer of Chen's question is ``yes" if one adds the Gorenstein condition.  More precisely, he showed that if $X$ is a Gorenstein minimal projective $3$-fold with at worst locally factorial terminal singularities, then the canonical degree of $X$ is at most $576$. Recently, the first and the second author improved Hacon's upper bound to $360$, and showed that equality holds if and only if $p_g(X)=4$, $q(X)=2$, $\chi(\omega_X)=5$, $K_X^3=360$ and $|K_X|$ is base point free. For $n<4$, the Miyaoka-Yau inequality plays a vital role in the proof. For $n\ge 4$ however, the Miyaoka-Yau inequality is not effective enough to give a universal bound to control $K_X^n$.

Another open question in this direction is to determine the positive degree of the canonical map. Progress in this direction for surfaces appears in \cite{Bea}, \cite{Tan}, \cite{Cas}, \cite{Per}, \cite{D-G1}, \cite{Rit1}, \cite{Rit2}, \cite{Rit3}, \cite{Yeung} and for $3$-folds in \cite{D-G2}, \cite{Cai}. Explicit examples are often constructed by taking abelian covers. Since the canonical degree may not be universally bounded, it is natural to ask the following question: Can one construct non-singular minimal projective $n$-folds whose canonical divisors define abelian covers over $\mathbb{P}^n$ with arbitrarily large degrees?

\begin{definition}
 Let $X$ be a minimal projective $n$-fold of general type with at worst locally factorial terminal singularities. If $|K_X|$ defines abelian cover over $\mathbb{P}^n$, then we call $X$ an \emph{abelian canonical $n$-fold}.
\end{definition}

In \cite{D-G1}, the first and the second authors obtained a complete classification of abelian canonical surfaces.  The universal upper bound for such surfaces is $16$. In \cite{D-G2}, they showed that for Gorenstein minimal projective $3$-folds of general type with at worst locally factorial terminal singularities, the upper bound of the canonical degrees of such $3$-folds is $32$.  Although we do not know if the canonical degree is universally bounded or not for higher dimensional projective varieties, there is evidence suggesting that the canonical degrees of non-singular abelian canonical $n$-folds may be universally bounded due to the fact that there are very strong restrictions on the defining data of abelian covers. In this paper, we show that the degrees of certain abelian covers are uniformly bounded. In particular, we give a negative answer to the above question. More precisely, we show that the canonical degrees of non-singular abelian canonical $n$-folds are universally bounded. We also construct two examples of non-singular minimal projective $4$-folds of general type with canonical degrees $81$ and $128$ in the last section.

\section{Basics on abelain covers}
The theory of covering is a very important tool in algebraic geometry. The cyclic covers of algebraic surfaces were first studied by Comessatti in \cite{Com}. Later, F. Catanese (\cite{Cat1}) studied smooth abelian covers in the case $(\mathbb{Z}_2)^{\oplus2}$. While R. Pardini analyzed the general case in \cite{Par1},  F. Catanese (\cite{Cat1}) pointed out that it is difficult to give the defining data of abelian covers by Pardini's method. Recently, the second author studied abelian covers of algebraic varieties from another point of view by calculating the normalization bases of the covering spaces which made the constructions more explicit (\cite{Gao}). In this section, we shall recall some basic definitions and results of abelian covers which will facilitate our subsequent discussion. As our goal is to determine the defining equations for covering spaces by explicit calculation, we use the method appearing in \cite{Gao}.

Let $X$ and $Y$ be projective algebraic varieties such that $X$ is normal, $Y$ is non-singular and $\varphi:X\to Y$ is an abelian cover associated with the abelian group $G\cong\mathbb Z_{n_1}\oplus\cdots\oplus\mathbb Z_{n_k}$, where $n_1|n_2\cdots|n_k$ (i.e., the function field $\mathbb C(X)$ of $X$ is an abelian extension of the rational function field $\mathbb C(Y)$ with Galois group $G$).

\begin{definition} \label{def}
Let $G\cong\mathbb Z_{n_1}\oplus\cdots\oplus\mathbb Z_{n_k}$. The data of an abelian cover over $Y$ with group $G$ consists of $k$ effective divisors $D_1$, $\cdots$, $D_k$, and $k$ linear equivalence relations
\[D_1\sim n_1L_1, \cdots, D_k\sim n_kL_k.\]
\end{definition}

Let $\mathscr{L}_i=\mathscr{O}_Y(L_i)$ and $f_i$ be the defining
equation of $D_i$, i.e., $D_i=\text{div}(f_i)$, where $f_i\in H^0(Y,
\mathscr{L}_i^{n_i})$. Denote by
$\textbf{V}(\mathscr{L}_i)=\textbf{Spec}S(\mathscr{L}_i)$ the
line bundle corresponding to $\mathscr{L}_i$, where
$S(\mathscr{L}_i)$ is the sheaf of the symmetric $\mathscr{O}_Y$
algebra, and let $z_i$ be the fiber coordinate of
$\textbf{V}(\mathscr{L}_i)$. Then the abelian cover can be realized
by the normalization of the variety $V$ defined by the system of equations
\begin{equation}\label{eqns}
z_1^{n_1}=f_1, \cdots, z_k^{n_k}=f_k.
\end{equation}
The above paragraph is summarized schematically via the following diagram:

\begin{diagram}
X     & \rTo^{\text{normalization}}   & V & \rdTo^f \rInto  &\oplus_{i=1}^k\textbf{V}(\mathscr{L}_i)\\
      & \rdTo(4,2)_\varphi            &   & \rdTo           &\dTo^p\\
      &                               &   &                 &Y .\\
\end{diagram}
We often make the abuse of saying $X$ is defined by equations (\ref{eqns}), although it should be clear from the context that $X$ is in fact the normalization of the solution $V$ of these equations.

We now list some useful results which will be crucial later on.
\begin{theorem}\label{F}\emph{(\cite{Gao})} Denote by $[Z]$ the integral part of a
$\mathbb Q$-divisor $Z$, and $-L_g=-\sum\limits_{i=1}^k g_i{L}_i+
\left[\sum\limits_{i=1}^{k}\frac{g_i}{n_i}D_i\right]$, where $g=(g_1,\cdots, g_k)\in G$.Then
\begin{eqnarray}\label{eqn1}
\varphi_*\mathcal{O}_X&=\bigoplus\limits_{{g\in G}}\mathcal
O_Y(-L_g).
\end{eqnarray}
\end{theorem}

We remark that equation (\ref{eqn1}) implies the decomposition of $\varphi_*\mathcal{O}_X$ is completely determined by the data of an abelian cover.

\begin{corollary}\label{h}\emph{(\cite{Gao})}
If $X$ is non-singular, $D$ is a divisor on $Y$, then
$$h^i(X, \varphi^*\mathcal{O}_Y(D))=\sum\limits_{g\in G}h^i(Y, \mathcal{O}_Y(D-L_g)).$$
\end{corollary}

The following result will be used to calculate the ramification
index, since the branching  locus is uniformly ramified for an abelian cover.

\begin{theorem} \label{branch}\emph{(\cite{Gao})}
Let $P$ be an irreducible and reduced hypersurface in $Y$,
let $\bar P=\pi^{-1}(P)$ be the reduced preimage of $P$ in $X$, and
let $a_i$ be the multiplicity of $P$ in $D_i=\textrm{div}(f_i)$.
Then
$$
\pi^*P=\dfrac{|G|}{d_P}\bar P,
$$
where
$$ d_P= \gcd\left(\,|G|, \ |G|\dfrac{a_1}{n_1}
,\ \cdots, \ |G|\dfrac{a_k}{n_k}  \, \right)$$
 is the number of points in the preimage of a generic point on
 $P$.
\end{theorem}

\section{abelian canonical $n$-folds}
Let $X$ and $Y$ be projective algebraic varieties, with $X$ normal and $Y$ non-singular, and let $\varphi:X\to Y$ be an abelian cover associated to the abelian group $G\cong\mathbb Z_{n_1}\oplus\cdots\oplus\mathbb Z_{n_k}$, where $n_1|n_2\cdots|n_k$. Using the notations after Definition \ref{def}, we have that $X$ is the normalization of the $n$-fold defined by
\begin{equation}\label{equa}
z_1^{n_1}=f_1, \cdots, z_k^{n_k}=f_k.
\end{equation}

We say that an abelian cover $\varphi:X\to Y$ is \emph{totally ramified} if the inertia subgroups
of the divisorial components of the branch locus of $\varphi$ generate $G$, or, equivalently,
if $\varphi$ does not factorize through a cover $X'\to Y$ that is $\acute{e}$tale over $Y$. Note that if $Cl(Y)$ has no torsion, then every connected abelian cover of $Y$ is totally ramified.

\begin{lemma}\label{contr}
With notations as above, suppose that the cover is totally ramified and $D_1, D_2, \cdots D_m$ are the irreducible components of the branch loci of the abelian cover $\varphi$ with ramification indices $r_1, r_2, \cdots, r_m$ respectively.  Then $n_k\leqslant \prod_{i=1}^m r_i$.
\end{lemma}
\begin{proof}
Suppose $n_t=p_1^{e_{t_1}} \cdots p_s^{e_{t_s}}$ , $1\leqslant t\leqslant k$ is the prime decomposition of $n_t$. Without loss of generality, we only need to show that there exist $D_j$ such that its ramification index $r_j$ is divisible by $p_s^{e_{k_s}}$. Otherwise, for any irreducible component of the branch locus, we would have for the ramification index

\[\frac{|G|}{d}=p_1^{e_1}\cdots p_s^{e_{s}}, \] such that $e_s<e_{k_s}$, where $$ d= \gcd\left(\,|G|, \ |G|\dfrac{a_1}{n_1}
,\ \cdots, \ |G|\dfrac{a_k}{n_k}  \, \right),$$ and $a_i$ is the multiplicity of the irreducible component in $\textrm{div}(f_i)$ by Theorem \ref{branch}. Thus \[d=p_1^{\sum_{j=1}^{k}e_{j_1}-e_1} \cdots  p_s^{\sum_{j=1}^{k}e_{j_s}-e_s}, \] which yields \[d\  |\  p_1^{\sum_{j=1}^{k-1}e_{j_1}} \cdots  p_s^{\sum_{j=1}^{k-1}e_{j_s}}a_k.\]
Hence $p_s^{e_{k_s}-e_s}\ |\  a_k$, which implies equations (\ref{equa}) split and so $X$ is not irreducible, a contradiction.
\end{proof}

\begin{theorem}
Let $Y$ be a non-singular projective $n$-fold such that the divisor class group $Cl(Y)$ torsion free, and fix both a divisor $L$ and an ample divisor $A$ on $Y$. Consider the following set of abelian covers
 \[\mathscr{C}_{L}:=\{\varphi: X\rightarrow Y\ |\  X\  \text{is nonsingular and}\  K_X=\varphi^*L\}.\]
 Then there exists a constant $C_{L,A}$ such that for any $\varphi\in \mathscr{C}_{L}$, deg $\varphi\leqslant C_{L,A}$.
\end{theorem}

\begin{proof}
Since Cl($Y$) is torsion free, all $D_i$'s in the defining data of each abelian cover over $Y$ as in Definition \ref{def} are non-zero. Suppose that $D_1, D_2, \cdots, D_m$ are the irreducible components of the branch locus of an abelian cover $\varphi\in \mathscr{C}_L$ with the ramification indices $r_1, r_2, \cdots, r_m$, where $r_i\geqslant 2, i=1,\cdots m$. Fix an ample divisor $A$ on $Y$ and write $d_i=D_iA^{n-1}$.  By the Hurwitz formula,

\begin{equation}\label{est}
K_X=\varphi^*(K_{Y}+\sum_{i=1}^{m}\frac{r_i-1}{r_i}D_i).
\end{equation}
On the other hand, $K_X=\varphi^*(L)$, thus
\begin{equation}
(L-K_Y)A^{n-1}=\sum_{i=1}^{m}\frac{r_i-1}{r_i}d_i,
\end{equation}
i.e.,
\begin{equation}\label{est2}
\ \sum_{i=1}^{m}\frac{1}{r_i}d_i=m-(L-K_Y)A^{n-1}.\end{equation}
We then have \[(L-K_Y)A^{n-1}< m \leqslant\sum_{i=1}^{m} d_i\leqslant 2\cdot(L-K_Y)A^{n-1}.\]
Without loss of generality suppose $r_1 \leqslant r_2 \leqslant \cdots \leqslant r_m$.  Then from (\ref{est2}) we have
\[\frac{1}{r_1}\sum_{i=1}^{m}d_i\geqslant m-(L-K_Y)A^{n-1},\] thus \[r_1\leqslant \sum_{i=1}^{m}d_i\leqslant 2\cdot (L-K_Y)A^{n-1}=:C_1.\] For fixed $m$, $r_1$ and $d_1$, \[\frac{1}{r_2}\sum_{i=2}^{m}d_i\geqslant \sum_{i=2}^{m}\frac{1}{r_i}d_i=m-(L-K_Y)A^{n-1}-\frac{d_1}{r_1},\] which yields

\begin{equation} r_2\leqslant \frac{2\cdot (L-K_Y)A^{n-1}}{m- (L-K_Y)A^{n-1}-\frac{d_1}{r_1}}. \end{equation}
As such, there exists a constant $C_2$ with $r_2\leqslant C_2$. By induction, there exists a constant $C_m$ such that $r_m\leqslant C_m$. By Lemma \ref{contr}, $n_k\leqslant \prod_{i=1}^m r_i\leqslant \prod_{i=1}^m C_i$, and since $n_k$, $m$ and $\sum_{i=1}^m d_i$ are finite, the number of the equations (\ref{equa}) is also finite, say $w$. So for any $\varphi\in \mathscr{C}_L$, deg $\varphi$ is bounded by the constant

$$C_{L,A}:=(\prod_{i=1}^m C_i)^w. $$
\end{proof}

Now let $X$ be a minimal projective $n$-fold of general type with at worst locally factorial terminal singularities. As we have defined in  \S\ref{secintro}, if $|K_X|$ defines an abelian cover over $\mathbb{P}^n$ then we call $X$ an abelian canonical $n$-fold.

\begin{corollary}
If $X$ is an abelian canonical $n$-fold, then there exists a constant $C(n)$ depending only on $n$ such that the canonical degree of $X$ is universally bounded by $C(n)$.
\end{corollary}
\begin{proof}
Take $L=H$ in the theorem above, where $H$ corresponds to a hyperplane in $\mathbb{P}^n$.
\end{proof}

\begin{proposition}\label{decom} Let $X$ be a non-singular abelian canonical $n$-fold and $\varphi$
be the finite abelian cover of degree $d$ over $\mathbb{P}^n$ defined by $|K_X|$, then $c_1((\varphi_*\mathcal{O}_X)^{\vee})=\frac{n+2}{2}\cdot d$ and
\[\varphi_*\mathcal{O}_X=\mathcal{O}_{\mathbb{P}^n}\oplus\mathcal{O}_{\mathbb{P}^n}(-n-2)\oplus(\mathcal{O}_{\mathbb{P}^n}(-2)\oplus\mathcal{O}_{\mathbb{P}^n}(-n))^{\oplus
k_2}\oplus(\mathcal{O}_{\mathbb{P}^n}(-3)\]
\[\oplus\mathcal{O}_{\mathbb{P}^n}(-n+1))^{\oplus
k_3}\oplus\cdots\oplus(\mathcal{O}_{\mathbb{P}^n}(-t)\oplus\mathcal{O}_{\mathbb{P}^n}(-n-2+t))^{\oplus
k_t}\oplus\cdots,\] where $\mathcal{O}_{\mathbb{P}^n}(-t)$  appears the same number of  times as $\mathcal{O}_{\mathbb{P}^n}(-n-2+t))$ in the direct sum.
\end{proposition}

\begin{proof}
If $\varphi$ is a finite abelian cover, $\varphi_*\mathcal{O}_X$ is a
direct sum of line bundles by Theorem \ref{F}, i.e.
$$\varphi_*\mathcal{O}_X\cong\mathcal{O}_{\mathbb{P}^n}\oplus
\bigoplus_{i=1}^{d-1}\mathcal{O}_{\mathbb{P}^n}(-l_i).$$ Now assume
$0<l_1\leqslant l_2\leqslant \cdots \leqslant l_{d-1}$. By relative duality,
\[\varphi_*\omega_X\cong(\varphi_*\mathcal{O}_X)^{\vee}\otimes\omega_{\mathbb{P}^n}\cong(\varphi_*\mathcal{O}_X)^{\vee}\otimes\mathcal{O}_{\mathbb{P}^n}(-n-1),\] and since
$\omega_X=\varphi^*(\mathcal{O}_{\mathbb{P}^n}(1))$, by the projection formula we have \[\varphi_*\omega_X\cong\varphi_*\varphi^*(\mathcal{O}_{\mathbb{P}^n}(1))\cong\mathcal{O}_{\mathbb{P}^n}(1)\otimes\varphi_*\mathcal{O}_X,\] so that \[(\varphi_*\mathcal{O}_X)^{\vee}\cong\varphi_*\mathcal{O}_X\otimes\mathcal{O}_{\mathbb{P}^n}(n+2),\]
from which it follows $\mathcal{O}_{\mathbb{P}^n}(-t)$  appears the same number of  times as $\mathcal{O}_{\mathbb{P}^n}(-n-2+t)$ in the direct sum. Therefore $2c_1((\varphi_*\mathcal{O}_X)^{\vee})=(n+2)d$, i.e., $c_1((\varphi_*\mathcal{O}_X)^{\vee})=\frac{n+2}{2}\cdot d$, so we only need to show that $t\geqslant 2$. For this, it follows by the definition of an abelian canonical $n$-fold that
$$p_g(X)=h^0(K_X)=h^0(\varphi^*(\mathcal{O}_{\mathbb{P}^n}(1)))
=n+1,$$ thus by the projection formula we have \[h^0(\mathcal{O}_{\mathbb{P}^n}(1))+
\sum_{i=1}^{d-1}h^0(\mathcal{O}_{\mathbb{P}^n}(1-l_i))=n+1.\]
We then have
$$h^0(\mathcal{O}_{\mathbb{P}^n}(1-l_i))=0, \quad 1\le i \le d-1,$$ thus $l_i\geqslant 2$.
\end{proof}

Now suppose $\varphi: X\rightarrow \mathbb{P}^n$ is an abelian cover
associated to an abelian group $G\cong\mathbb
Z_{n_1}\oplus\cdots\oplus\mathbb Z_{n_k}$, such that $n_1|n_2|\cdots|n_k$. Then $X$ is the
normalization of the $n$-fold defined by
\begin{equation}
z_1^{n_1}=f_1=\prod_\alpha p_\alpha^{\alpha_1}, \cdots, z_k^{n_k}=f_k=\prod_\alpha p_\alpha^{\alpha_k},
\end{equation}
where the $p_{\alpha}$'s are coprime and $\alpha=(\alpha_1,\cdots,
\alpha_k)\in G$, $\alpha_1,\cdots,
\alpha_k<n_k$. Denote  by $x_\alpha$ the degree of $p_\alpha$,
$e_i=(0,\cdots,0,1,0,\cdots,0)\in G$ with $1\leq i \leq k$, and denote by $l_g$
the degree of $L_g$. The $x_\alpha$ and $l_g$ are then integers. We then have
\begin{eqnarray}\label{s1}
&n_il_{e_i}=\sum\limits_\alpha \alpha_ix_\alpha\quad i=1,\cdots k,\\\label{s2}
& l_g=\sum\limits_{i=1}^k g_i
l_{e_i}-\sum\limits_\alpha\left[\sum\limits_{i=1}^k
\frac{\displaystyle{g_i\alpha_i}}{\displaystyle
 {n_i}}\right]x_\alpha.
 \end{eqnarray}

\section{abelian canonical $4$-folds}
Since the canonical degrees of non-singular abelian canonical $4$-folds have universal bounds, it is interesting to find some examples of such $4$-folds such that the canonical degree is as close to the bound as possible.

By  Proposition \ref{decom} and equations (\ref{s1}) and (\ref{s2}), we construct two examples of nonsingular abelian canonical $4$-folds (cf. \cite{D-G2} Theorem 3.6).

\begin{example}
Let $X$ be the canonical abelian $4$-fold which is the normalization of the variety whose defining equations are given by
$$\left\{\begin{array}{l}
z_1^2=h_{1}h_{8}h_{9}h_{10}\\
z_2^2=h_{2}h_{8}h_{9}h_{11}\\
z_3^2=h_{3}h_{8}h_{9}h_{12}\\
z_4^2=h_{4}h_{8}h_{10}h_{11}\\
z_5^2=h_{5}h_{8}h_{10}h_{12}\\
z_6^2=h_{6}h_{9}h_{10}h_{11}\\
z_7^2=h_{7}h_{9}h_{10}h_{12},
\end{array}\right .$$
where the hyperplanes $H_i$'s defined by $h_i$'s  are normal crossing in $\mathbb{P}^4$.

First, we will show that $X$ is nonsingular.
Obviously, the possible singularities of $X$ lie in the preimages of the intersections of the branch locus.
Without loss of generality, we assume that $P$ is the intersection point of $H_8$, $H_9$, $H_{10}$ $H_{11}$. The cover is locally defined by the following equations at $P$:
$$z_1^2=xyu, z_2^2=xyt, z_3^2=xy, z_4^2=xut, z_5^2=xu, z_6^2=yut, z_7^2=yu.$$ After normalization, the cover $X$ is locally defined by
\[\bar{z}_1^2=x,\quad \bar{z}_2^2=t,\quad \bar{z}_3^2=u,\quad \bar{z}_4^2=y,\] thus $X$ is smooth at the preimages of $P$.

Similarly, it is easy to show that $X$ is non-singular on the preimage of the sets $H_i\cap H_j \cap H_k \cap H_l$, $H_i\cap H_j \cap H_k$ and $H_i\cap H_j$ for all $i,j,k,l$. So $X$ is nonsingular.

 Next, we want to calculate $l_g$ by (\ref{s2}). Note that in this example $x_i=$deg$(h_i)=1$, for $1\leq i \leq 12$. It is easy to check that $l_{e_i}=2$, for all $i$, by (\ref{s1}). For other $l_g$'s, we take $g=(1,1,0,0,0,0,0)$ and $g'=(1,1,1,1,1,1,1)$ for example.

\begin{equation}
\begin{split}
l_{g}&=l_{e_1}+l_{e_2}-\left[\frac{1+0}{2}\right]x_1-\left[\frac{0+1}{2}\right]x_2-\sum_{i=3}^7\left[\frac{0+0}{2}\right]x_i\\
&-\left[\frac{1+1}{2}\right]x_8-\left[\frac{1+1}{2}\right]x_9-\left[\frac{1+0}{2}\right]x_{10}-\left[\frac{0+1}{2}\right]x_{11}-\left[\frac{0+0}{2}\right]x_{12}\\
&=4-x_8-x_9\\
&=2,
 \end{split}
 \end{equation}

\begin{equation}
\begin{split}
l_{g'}&=\sum\limits_{i=1}^7 l_{e_i}-\sum_{i=1}^7\left[\frac{1}{2}\right]x_i-\left[\frac{1+1+1+1+1+0+0}{2}\right]x_8\\
&-\left[\frac{1+1+1+0+0+1+1}{2}\right]x_9-\left[\frac{1+0+0+1+1+1+1}{2}\right]x_{10}\\
&-\left[\frac{0+1+0+1+0+1+0}{2}\right]x_{11}-\left[\frac{0+0+1+0+1+0+1}{2}\right]x_{12}\\
&=14-\left[\frac{5}{2}\right]x_8-\left[\frac{5}{2}\right]x_9-\left[\frac{5}{2}\right]x_{10}-\left[\frac{3}{2}\right]x_{11}-\left[\frac{3}{2}\right]x_{12}\\
&=6.
 \end{split}
 \end{equation}

Finally, we have
\begin{equation}
\begin{split}
l_{(0,0,0,0,0,0,1)}&=l_{(0,0,0,0,0,1,0)}=l_{(0,0,0,0,0,1,1)}=l_{(0,0,0,0,1,0,0)}=l_{(0,0,0,0,1,0,1)}\\
=l_{(0,0,0,1,0,0,0)}&=l_{(0,0,0,1,0,1,0)}=l_{(0,0,0,1,1,0,0)}=l_{(0,0,0,1,1,1,1)}=l_{(0,0,1,0,0,0,0)}\\
=l_{(0,0,1,0,0,0,1)}&=l_{(0,0,1,0,1,0,0)}=l_{(0,0,0,0,1,0,1)}=l_{(0,0,0,0,1,1,0)}=l_{(0,0,1,1,0,0,1)}\\
=l_{(0,0,1,1,0,1,0)}&=l_{(0,1,0,0,0,0,0)}=l_{(0,1,0,0,0,1,0)}=l_{(0,1,0,0,1,0,1)}=l_{(0,1,0,0,1,1,0)}\\
=l_{(0,1,0,1,0,0,0)}&=l_{(0,1,0,0,0,0,1)}=l_{(0,1,0,1,0,1,0)}=l_{(0,1,1,0,0,0,0)}=l_{(0,1,1,0,0,1,1)}\\
=l_{(0,1,1,1,1,0,0)}&=l_{(1,0,0,0,0,0,0)}=l_{(1,0,0,0,0,0,1)}=l_{(1,0,0,0,0,1,0)}=l_{(1,0,0,0,1,0,0)}\\
=l_{(1,0,0,0,1,0,1)}&=l_{(1,0,0,1,0,0,0)}=l_{(1,0,0,1,0,1,0)}=l_{(1,0,1,0,0,0,0)}=l_{(1,0,1,0,0,0,1)}\\
=l_{(1,0,1,0,1,0,0)}&=l_{(1,1,0,0,0,0,0)}=l_{(1,1,0,0,0,1,0)}=l_{(1,1,0,1,0,0,0)}=2,
 \end{split}
 \end{equation}
\begin{equation}
\begin{split}
l_{(0,0,0,0,1,1,0)}&=l_{(0,0,0,0,1,1,1)}=l_{(0,0,0,1,0,0,1)}=l_{(0,0,0,1,0,1,1)}=l_{(0,0,0,1,1,0,1)}\\
=l_{(0,0,0,1,1,1,0)}&=l_{(0,0,1,0,0,1,0)}=l_{(0,0,1,0,0,1,1)}=l_{(0,0,1,1,0,0,0)}=l_{(0,0,1,1,0,1,1)}\\
=l_{(0,0,1,1,1,0,0)}&=l_{(0,0,1,1,1,1,0)}=l_{(0,1,0,0,0,0,1)}=l_{(0,1,0,0,0,1,1)}=l_{(0,1,0,0,1,0,0)}\\
=l_{(0,1,0,0,1,1,1)}&=l_{(0,1,0,1,1,0,0)}=l_{(0,1,0,1,1,0,1)}=l_{(0,1,1,0,0,0,1)}=l_{(0,1,1,0,0,1,0)}\\
=l_{(0,1,1,0,1,0,0)}&=l_{(0,1,1,0,1,1,0)}=l_{(0,1,1,1,0,0,0)}=l_{(0,1,1,1,0,0,1)}=l_{(1,0,0,0,1,1,0)}\\
=l_{(1,0,0,0,1,1,1)}&=l_{(1,0,0,1,0,0,1)}=l_{(1,0,0,1,0,1,1)}=l_{(1,0,0,1,1,0,1)}=l_{(1,0,0,1,1,1,0)}\\
=l_{(1,0,1,0,0,1,0)}&=l_{(1,0,1,0,0,1,1)}=l_{(1,0,1,1,0,0,0)}=l_{(1,0,1,1,0,1,1)}=l_{(1,0,1,1,1,0,0)}\\
=l_{(1,0,1,1,1,1,0)}&=l_{(1,1,0,0,0,0,1)}=l_{(1,1,0,0,0,1,1)}=l_{(1,1,0,0,1,0,0)}=l_{(1,1,0,0,1,1,1)}\\
=l_{(1,1,0,1,1,0,0)}&=l_{(1,1,0,1,1,0,1)}=l_{(1,1,1,0,0,0,1)}=l_{(1,1,1,0,0,1,0)}=l_{(1,1,1,0,1,0,0)}\\
=l_{(1,1,1,0,1,1,0)}&=l_{(1,1,1,1,0,0,0)}=l_{(1,1,1,1,0,0,1)}=3,
 \end{split}
 \end{equation}

\begin{equation}
\begin{split}
l_{(0,0,1,0,1,1,1)}&=l_{(0,0,1,1,1,0,1)}=l_{(0,0,1,1,1,1,1)}=l_{(0,1,0,1,0,1,1)}=l_{(0,1,0,1,1,1,0)}\\
=l_{(0,1,0,1,1,1,1)}&=l_{(0,1,1,0,1,0,1)}=l_{(0,1,1,0,1,1,1)}=l_{(0,1,1,1,0,1,0)}=l_{(0,1,1,1,0,1,1)}\\
=l_{(0,1,1,1,1,0,1)}&=l_{(0,1,1,1,1,1,0)}=l_{(0,1,1,1,1,1,1)}=l_{(1,0,0,0,0,1,1)}=l_{(1,0,0,1,1,0,0)}\\
=l_{(1,0,0,1,1,1,1)}&=l_{(1,0,1,0,1,0,1)}=l_{(1,0,1,0,1,1,0)}=l_{(1,0,1,0,1,1,1)}=l_{(1,0,1,1,0,0,1)}\\
=l_{(1,0,1,1,0,1,0)}&=l_{(1,0,1,1,1,0,1)}=l_{(1,0,1,1,1,1,1)}=l_{(1,1,0,0,1,0,1)}=l_{(1,1,0,0,1,1,0)}\\
=l_{(1,1,0,1,0,0,1)}&=l_{(1,1,0,1,0,1,0)}=l_{(1,1,0,1,0,1,1)}=l_{(1,1,0,1,1,1,0)}=l_{(1,1,0,1,1,1,1)}\\
=l_{(1,1,1,0,0,0,0)}&=l_{(1,1,1,0,0,1,1)}=l_{(1,1,1,0,1,1,0)}=l_{(1,1,1,0,1,1,1)}=l_{(1,1,1,1,0,1,0)}\\
=l_{(1,1,1,1,0,1,1)}&=l_{(1,1,1,1,1,0,0)}=l_{(1,1,1,1,1,0,1)}=l_{(1,1,1,1,1,1,0)}=4,
 \end{split}
 \end{equation}
$$l_{(0,0,0,0,0,0,0)}=0$$ and $$l_{(1,1,1,1,1,1,1)}=6.$$
Therefore, by Theorem \ref{F},
\[\varphi_*\mathcal{O}_X=\mathcal{O}_{\mathbb{P}^4}\oplus\mathcal{O}_{\mathbb{P}^4}(-6)\oplus\mathcal{O}_{\mathbb{P}^4}(-2)^{\oplus39}\oplus\mathcal{O}_{\mathbb{P}^4}(-3)^{\oplus48}\oplus\mathcal{O}_{\mathbb{P}^4}(-4)^{\oplus39}.\] Thus, by Corollary \ref{h}, $$p_g(X)=5, \quad q=h^{2,0}=h^{3,0}=0, \quad \chi(\omega_X)=6$$
and by Hurwitz formula, $K_X=\varphi^*\mathcal{O}_{\mathbb{P}^4}(1)$. So $X$ is minimal and the canonical degree of $X$ is $K_X^4=128$.
\end{example}

Similarly, we have the following example.
\begin{example}
Let $X$ be the canonical abelian $4$-fold which is the normalization of the variety whose defining equations are given by
$$\left\{\begin{array}{l}
z_1^3=h_{1}h_{5}^2h_{6}^2h_{7}\\
z_2^3=h_{2}h_{6}^2h_{7}h_{8}^2\\
z_3^3=h_{3}h_{7}h_{8}^2h_{9}^2\\
z_4^3=h_{4}h_{5}^2h_{7}h_{9}^2,
\end{array}\right.$$
where the hyperplanes $H_i$'s defined by $h_i$'s intersect with normal crossings in $\mathbb{P}^4$. 

The arguments of smoothness of $X$ and the calculation of $l_g$'s are as similar as the above example. Finally, we have
\begin{equation}
\begin{split}
l_{(0,0,0,1)}&=l_{(0,0,0,2)}=l_{(0,0,1,0)}=l_{(0,0,1,2)}=l_{(0,0,2,0)}=l_{(0,0,2,1)}=l_{(0,1,0,0)}\\
=l_{(0,1,2,0)}&=l_{(0,2,0,0)}=l_{(0,2,1,0)}=l_{(1,0,0,0)}=l_{(1,0,0,2)}=l_{(1,2,0,0)}=l_{(1,2,1,2)}\\
=l_{(2,0,0,0)}&=l_{(2,0,0,1)}=l_{(2,1,0,0)}=l_{(2,1,2,1)}=2,
 \end{split}
 \end{equation}
 \begin{equation}
\begin{split}
l_{(0,0,1,1)}&=l_{(0,0,2,2)}=l_{(0,1,0,2)}=l_{(0,1,1,0)}=l_{(0,1,1,1)}=l_{(0,1,1,2)}=l_{(0,1,2,1)}\\
=l_{(0,2,0,1)}&=l_{(0,2,0,2)}=l_{(0,2,1,1)}=l_{(0,2,1,2)}=l_{(0,2,2,0)}=l_{(1,0,0,1)}=l_{(1,0,1,1)}\\
=l_{(1,0,1,2)}&=l_{(1,0,2,0)}=l_{(1,0,2,1)}=l_{(1,1,0,0)}=l_{(1,1,0,1)}=l_{(1,1,0,2)}=l_{(1,1,1,0)}\\
=l_{(1,1,1,1)}&=l_{(1,1,1,2)}=l_{(1,1,2,0)}=l_{(1,1,2,1)}=l_{(1,1,2,2)}=l_{(1,2,0,1)}=l_{(1,2,0,2)}\\
=l_{(1,2,1,0)}&=l_{(1,2,1,1)}=l_{(1,2,2,1)}=l_{(2,0,0,2)}=l_{(2,0,1,0)}=l_{(2,0,1,1)}=l_{(2,0,2,0)}\\
=l_{(2,0,2,1)}&=l_{(2,1,0,1)}=l_{(2,1,1,0)}=l_{(2,1,1,1)}=l_{(2,1,1,2)}=l_{(2,1,2,0)}=l_{(2,2,0,0)}\\
=l_{(2,2,1,1)}&=3,
 \end{split}
 \end{equation}
 \begin{equation}
\begin{split}
l_{(0,1,0,1)}&=l_{(0,1,2,2)}=l_{(0,2,2,1)}=l_{(0,2,2,2)}=l_{(1,0,1,0)}=l_{(1,0,2,2)}=l_{(1,2,2,0)}\\
=l_{(1,2,2,2)}&=l_{(2,0,1,2)}=l_{(2,0,2,2)}=l_{(2,1,0,2)}=l_{(2,1,2,2)}=l_{(2,2,0,1)}=l_{(2,2,0,2)}\\
=l_{(2,2,1,0)}&=l_{(2,2,1,2)}=l_{(2,2,2,0)}=l_{(2,2,2,1)}=4,
 \end{split}
 \end{equation}
 $$l_{(0,0,0,0)}=0$$ and $$l_{(1,1,1,1)}=6.$$
So
\[\varphi_*\mathcal{O}_X=\mathcal{O}_{\mathbb{P}^4}\oplus\mathcal{O}_{\mathbb{P}^4}(-6)\oplus\mathcal{O}_{\mathbb{P}^4}(-2)^{\oplus18}\oplus\mathcal{O}_{\mathbb{P}^4}(-3)^{\oplus43}\oplus\mathcal{O}_{\mathbb{P}^4}(-4)^{\oplus18}.\] Thus, by Corollary \ref{h}, $p_g(X)=5$, $q=h^{2,0}=h^{3,0}=0$, $\chi(\omega_X)=6$ and by Hurwitz formula, $K_X=\varphi^*\mathcal{O}_{\mathbb{P}^4}(1)$. So $X$ is minimal and the canonical degree of $X$ is $K_X^4=81$.  
\end{example}

\section*{Acknowledgements}
Both authors would like to thank the referee for comments which lead to a generalization of the main theorem, and for other useful suggestions which lead to a considerable improvement of the paper.


\begin{thebibliography}{EVW}


\bibitem [Bea]{Bea} Beauville, A.: \textit{L'application canonique pour les surfaces de type g$\acute{e}$n$\acute{e}$ral},
Invention Math. \textbf{55}, 121-140(1979).

\bibitem [Cai]{Cai} Cai, J.: \text{Degree of the canonical map of a Georenstein 3-fold of general type}, Proc. AMS, \textbf{136}, (2008), No. 5, 1565-1574.

\bibitem [Cas]{Cas} Casnati, G.: \textit{Covers of algebraic varieties. II. Covers of degree 5 and construction of surfaces.}, J. Algebraic Geom. \textbf{5} (1996), no. 3, 461-477.

\bibitem [Cat1]{Cat1} Catanese, F.: \textit{On the moduli spaces of surfaces of general type}, J. Differential Geom., \textbf{19} (1984), 483-515.

\bibitem [Ch]{Ch} Chen, M.: \textit{Weak Boundedness Theorems for Canonically Fibered Gorenstein Minimal 3-Folds}, Proc. Amer. Math. Soc., \textbf{133} (2005), No. 5, 1291-1298.

\bibitem [Com]{Com} Comessatti, A.: \textit{Sulle superfici multiple cicliche}, Rend. Sem. Mat. Univ. Padova, \textbf{1} (1930), 1-45.

\bibitem [D-G1]{D-G1} Du, R.; Gao, Y.: \textit{Canonical maps of surfaces defined by abelian covers}, Asian J. Math. \textbf{18} (2014), No. 1, 219-228.

\bibitem [D-G2]{D-G2} Du, R.; Gao, Y.: \textit{On the canonical degrees of Gorenstein threefolds of general type}, Geom. Dedicata \textbf{18}, (2016), No. 1, 123-130.

\bibitem [Gao]{Gao} Gao, Y.: \textit{A note on finite abelian cover}, Science in China, Series A: Mathematics \textbf{54} (2011), 1333-1342.

\bibitem [Ha]{Ha} Hacon, C.: \textit{On the degree of the canonical maps of 3-folds}, Proc. Japan Acad. \textbf{80} Ser. A (2004) 166-167.

\bibitem [Par1]{Par1} Pardini, R.: \textit{Abelian covers of algebraic varieties}, J. Reine Angew. Math. \textbf{417}(1991), 191-213.

\bibitem [Per]{Per} Persson, U.: \textit{Double coverings and surfaces of general type.}, In: Olson,L.D.(ed.) Algebraic geometry.(Lect. Notes Math., vol.732, pp.168-175) Berlin Heidelberg New York: Springer 1978.

\bibitem [Rit1]{Rit1} Rito, C.: \textit{New canonical triple covers of surfaces.} Proc. AMS, Vol. 143, No. 11 (2015), 4647-4653.

\bibitem [Rit2]{Rit2} Rito, C.: \emph{A surface with $q=2$ and canonical map of degree $16$}, arXiv:1506.05987.

\bibitem [Rit3]{Rit3} Rito, C.: \emph{A surface with canonical map of degree $24$}, arXiv:1509.04132.

\bibitem [Tan]{Tan}  Tan, S.-L.: \textit{Surfaces whose canonical maps are of odd degrees}.
Math. Ann. \textbf{292} (1992), no. 1, 13-29.

\bibitem [Yeung]{Yeung} Yeung, S.: \textit{A surface of maximal canonical degree}, Math. Ann., online, DOI: 10.1007/s00208-016-1450-x.
\end{thebibliography}
\end{document}